\documentclass[a4paper,11pt,reqno,english]{amsart}

\usepackage{anysize,color}
\usepackage[utf8]{inputenc}
\usepackage[dvipsnames]{xcolor}
\usepackage[colorlinks=true,urlcolor=blue,linkcolor={magenta},citecolor={orange}]{hyperref}  
\usepackage{float}
\usepackage{amsfonts,amssymb,amsmath}
\usepackage{amsthm}    
\usepackage{url}
\usepackage{paralist}
\usepackage[mathscr]{euscript}
\usepackage{charter}
\usepackage{bbold}

\newtheorem{theorem}{Theorem}[section]
\newtheorem{proposition}[theorem]{Proposition}
\newtheorem{lemma}[theorem]{Lemma} 
\newtheorem{corollary}[theorem]{Corollary} 
\theoremstyle{definition}
\newtheorem{definition}[theorem]{Definition}
\newtheorem{example}[theorem]{Example}

\newcommand\R{\mathbb R}

\author[Blagojevi\'c]{Matija Blagojevi\'{c}}
\thanks{ 
	The research by Matija Blagojevi\'c leading to these results has received funding from Studienstiftung des deutschen Volkes through their Expos\'{e}-Stipendium program.
}
\email{blagojevic@zib.de}
\address{Zuse Institut, 14195 Berlin, Germany}
\author[Sch\"{u}tte]{Christof Sch\"{u}tte}
\email{schuette@zib.de}
\address{Zuse Institut, 14195 Berlin, Germany}
\title{Topology of the Generalized Nash Equilibrium Problem}

\begin{document}
\begin{abstract}
	The Generalized Nash Equilibrium Problem refers to the question of the existence of a Nash equilibrium in an abstract economy. This model is due to Kenneth J. Arrow and G\'erard Debreu in their pioneering work from 1954. An abstract economy is an extension of John Nash's original concept of a non-cooperative game from the 1950s. Here players selfishly seek to maximize their profits which may depend on the others' choices. The novelty of an abstract economy is that the players may now mutually constrain each other in their decision-making. A Nash equilibrium is reached when no player alone can increase his profit by a unilateral change of strategy. Abstract economies have found widespread applications from welfare economy, economic analysis and policymaking to constrained optimization, partial differential equations and optimal allocation.

	We generalize Leigh Tesfatsion's Nash equilibrium existence result for non-cooperative games to abstract economies without resorting to commonly employed convexity assumptions, thereby re-proving all known Nash equilibrium existence results as special cases. A key element of our proof is the translation of the Generalized Nash Equilibrium Problem into the question of the existence of a coincidence of two particularly defined maps. We then apply a coincidence result from algebraic topology due to Samuel Eilenberg and Deane Montgomery in 1946.
	Additionally, we provide examples of abstract economies which satisfy the assumptions of our main result, but due to lacking convexity assumptions, do not satisfy the assumptions of other classical results.

\end{abstract}
\maketitle
\section{Introduction}

The notion of an ``abstract economy'' can be traced back to 1954 and Kenneth J. Arrow and Gerard Debreu's article~\cite{ArrowDebreu1954}.
The authors introduced a model which would subsequently become the namesake of an abstract economy and proved the existence of an equilibrium therein.
Both of the authors, independently, received Nobel Prizes in Economics for this work.

Since then, the notion of an abstract economy has been expanded by many authors, like Lionel W. McKenzie in~\cite{McKenzie}.
He significantly weakened many of Arrow and Debreu's assumptions and introduced a new mathematical framework, not too dissimilar from what we use today.
At present, abstract economies are often referred to as Generalized Nash Equilibrium Problems (GNEPs for short), highlighting their relationship to John Nash's work from 1951 in~\cite{Nash1951}.
Due to the similarity of concepts, equilibria in abstract economies are typically referred to as Nash equilibria.
For a more complete and modern presentation of Nash's non-cooperative games as well as abstract economies see also Tatsuro Ichiishi's classical book~\cite{Ichiishi_1983}.

\medskip Let us now make the model precise.
Let $N\geq 1$ be a fixed integer representing the number of players or participants in our economy, labelled with the integers $\{1,\dots,N\}$.
Each player $i$ has a \textit{decision space} $E_i$, typically assumed to be a metric space.
We denote by $E$ the cartesian product $E:= E_1\times\cdots\times E_N$ and call it \textit{global decision space}.
Player $i$ ``plays'' the abstract economy by selecting a decision $x_i\in E_i$.
An $N$-tuple $(x_1,\dots,x_n)\in E$ is called a \textit{global decision}.

Within the framework of an abstract economy, there may be restrictions imposed on the players.
Though the players choose their decision independently of each other, not every global decision is permissible.
A common reason for imposing restrictions could be, say in applications relating to physics, that certain global decisions would break physical limitations of the considered system.
This happens, for example, in gas network modelling, where gas pipes have a fixed maximum rated pressure, so that the total production of the companies in the network may not exceed this rated pressure at any point.
Mathematically this restriction manifests as a constraint on the solutions of certain partial differential equations, consult for example~\cite{Hintermueller2015}.

The restrictions on the players are modelled via a collection of $N$ set-valued maps given for $1\leq i \leq N$ by $X_i\colon E\longrightarrow 2^{E_i}$.
Here the notation $2^A$ stands for the power set of $A$.
We do not make a notational distinction between functions and set-valued maps, because we always accurately reflect the codomain of our maps.
The pointwise cartesian product of the set-valued maps $X_1,\dots,X_N$ is denoted by $X\colon E\longrightarrow 2^E$ and given by $X(x)=X_1(x)\times\cdots\times X_N(x)$.
The restrictions on the players are incorporated in our framework by requiring that for all $1\leq i\leq N$, the $i$-th player's selected decision $x_i\in E_i$ satisfies $x_i\in X_i(x)$, where $x:=(x_1,\dots,x_N)$.

Note that in this definition $X_i$ may also depend on the choice of the $i$-th player. This permissiveness can be seen as somewhat unnatural: How can the restrictions on player $i$ also simultaneously depend on his own choice? While this assumption does not cause any issues for our main result, many authors have required that $X_i(x)$ is independent of the $i$-th coordinate of $x$, and this assumption will be useful to us again at a later point.

Speaking of the set-valued maps $X_i$, the following assumptions are obviously necessary to be able to conclude anything productive:
\begin{compactitem}[\quad --]
	\item There exists $x\in E$ such that $x\in X(x)$. Otherwise, the players cannot all simultaneously satisfy their restrictions, and the abstract economy cannot be played.
	\item For all $1\leq i\leq N$ and for all $x\in E$ the set $X_i(x)\neq\emptyset$ is nonempty. Otherwise, a player may not have any admissible decisions, and once again the abstract economy cannot be played.
\end{compactitem}
So henceforth, we will require every abstract economy to satisfy these two (very mild) assumptions.

\medskip
The next part of the data in an abstract economy are the continuous \textit{payoff} or \textit{utility functions} denoted by $\theta_i\colon\operatorname{Graph}(X_i)\longrightarrow \R$.
These profit functions describe the preferences of the players of the abstract economy, who seek to maximize their personal profits.
Here, $\operatorname{Graph}(X_i)=\{(x,y)\in E\times E_i : y\in X_i(x)\}$ is the graph of the set-valued map $X_i$.
As before, one may notice that $E_i$ appears twice in the product $E\times E_i = E_1\times\cdots\times E_i\times\cdots\times E_N\times E_i$.
Some sources, for example~\cite{FacchineiKanzow2007}, require that $\theta_i$ does not depend on the $i$-th factor in $E$, but this special case is already included in our work.
Furthermore, often $\theta_i$ may be defined or trivially extended from $\operatorname{Graph}(X_i)\subseteq E\times E_i$ to all of $E\times E_i$.
In the definition we prefer to leave the domain of $\theta_i$ as $\operatorname{Graph}(X_i)$ to emphasize that the set-valued maps $X_i$ impose inviolable restrictions on the players.

Summarizing, we can define an $N$-player abstract economy as the collection of the aforementioned data $\mathcal{E}=(E_i,X_i,\theta_i\colon 1\leq i\leq N)$, where

\begin{compactitem}[\quad --]
	\item $(E_1,\dots,E_N)$ is a collection of non-empty metric spaces, with $E:=E_1\times \cdots\times E_N$,
	\item $(X_1,\dots,X_N)$  is a collection of set-valued maps where $X_i\colon E\longrightarrow 2^{E_i}$, for $1\leq i\leq N$, with $X\colon E\longrightarrow 2^E$ being the set-valued map $X(x):= X_1(x)\times\cdots\times X_N(x)$, and
	\item  $(\theta_1,\dots,\theta_N)$ is a collection of continuous functions $\theta_i\colon\operatorname{Graph}(X_i)\longrightarrow\R$, for $1\leq i\leq N$.
\end{compactitem}

The setup of non-cooperative games of Nash's original work~\cite{Nash1951} can be recovered from this definition: In our language, he considered $E_i$ to be the convex hull of finitely many points in finite-dimensional Euclidean space, so a convex polytope.
The fact that a non-cooperative does not pose restrictions on the global strategies is represented in our model by setting $X_i$ to be the constant set-valued map $X_i(x)=E_i\in 2^{E_i}$.
Thus, non-cooperative games are included as a special case of an abstract economy.

\medskip
An equilibrium of an abstract economy $\mathcal{E}$, often called \textit{social equilibrium} or \textit{Nash equilibrium},
is defined as a global decision $x=(x_1,\dots,x_N)\in E$ such that two conditions are fulfilled:
\begin{compactitem}[\quad --]
	\item for all $1\leq i \leq N$ the decision $x_i$ satisfies the $i$-th restriction, this means $x_i\in X_i(x)$, and
	\item for all $1\leq i \leq N$ and for all of player $i$'s admissible deviations $x_i'\in X_i(x)$ it holds that
	\begin{displaymath}
		\theta_i(x,x_i) \geq \theta_i(x,x_i'),
	\end{displaymath}
	meaning the payoff which player $i$ receives for any unilateral deviation away from $x$ is no greater than the payoff player $i$ would receive from the equilibrium $x$.
\end{compactitem}
As expected, in the special case of non-cooperative games, this definition exactly agrees with the notion of a Nash equilibrium of a non-cooperative game.
Indeed, the first condition is void, and the second condition is exactly the Nash optimality condition.

One of the reasons for our interest in Nash equilibria becomes obvious once we translate the conditions above into words. Indeed, in a Nash equilibrium no player by himself is incentivized to change his decision, since the second condition says that any deviation of a particular player could only lead to an equal or worse payoff for exactly that same player. In other words: No player can make a profit by changing their decision alone. Although there is no notion of time in an abstract economy, this reasoning can be taken as an argument for a Nash equilibrium being a ``stable'' state in the economy, justifying the name ``equilibrium''.

\medskip
Even though the applications of abstract economies were initially envisioned to relate to economics, Patrick T. Harker in his seminal publication~\cite{Harker1989} established a connection between GNEPs and quasi-variational inequalities, opening the door to applications of GNEPs in constrained optimization and operations research.
Since then, in the late $20^{\text{th}}$ and early $21^{\text{st}}$ century, GNEPs have found applications in diverse fields: From machine learning and prediction games~\cite{SchefferMicheal2009} to resource allocations in cloud systems~\cite{CloudSystems2011} and telecommunications systems, as presented in~\cite{FacchineiKanzow2007}.

\section{Statement of the Main Result}

\medskip
To get a better understanding of Nash equilibria of abstract economies while simultaneously illustrating yet another reason for our interest in them, we introduce a particular set-valued map which helps us capture the interests of the players: the so-called best-reply correspondence.
For a player $i$, where $1\leq i\leq N$, define
\begin{alignat*}{3}
	\varphi_i\colon & E\longrightarrow\mathbb{R}\cup\{\infty\}, & \qquad & x \longmapsto  \sup\{\theta_i(x, y_i) \colon y_i\in X_i(x)\},            \\
	\Phi_i\colon    & E\longrightarrow 2^{E_i},                 & \qquad & x \longmapsto    \{y_i\in X_i(x)\colon \theta_i(x,y_i) = \varphi_i(x)\}.
\end{alignat*}
With this definition, the set valued map $\Phi_i$ associates to each global decision $x$ the set of optimal responses of player $i$, which is why we say that $\Phi_i$ is the \textit{best-reply correspondence} of player $i$.
Note that it may well be that for some global decisions $x$ the value $\Phi_i(x)=\emptyset$ is empty.
As usual, we call the set-valued map $\Phi\colon E\longrightarrow 2^{E}$ defined as the cartesian product $\Phi(x)=\Phi_1(x)\times\cdots\times\Phi_N(x)$ the \textit{global best-reply correspondence}. We remark that the functions $\varphi_i$ are named \textit{marginal functions} of $\theta_i$.

Let us re-examine the Nash equilibrium condition. Note that for all $1\leq i \leq N$ the following equivalence holds
\begin{displaymath}
	(\forall x_i'\in X_i(x)) \ \theta_i(x,x_i) \geq \theta_i(x,x_i') \quad \Longleftrightarrow \quad \theta_i(x,x_i) = \sup\{\theta_i(x,x_i')\colon x_i'\in X_i(x) \} = \phi_i(x).
\end{displaymath}
Now the statement on the right-hand side is by definition equivalent to $x_i\in\Phi_i(x)$. Consequently, we obtain that $x$ is a Nash equilibrium if and only if $x\in\Phi(x)$. Once again, this fact seems almost obvious if we formulate it in words: In a Nash equilibrium, every player is responding optimally to everyone else's decision. This is another reason we care about the existence of Nash equilibria.

\medskip
For any set valued map $F\colon M\longrightarrow 2^{M}$, a point $m\in M$ such that $m\in F(m)$ is called a \textit{fixed point} of the set-valued map $F$.
So a natural approach to prove the existence of a Nash equilibrium would be to employ some sort of fixed-point theorem.
Historically, the famous Kakutani's fixed-point theorem was often used, for example in Ichiishi's book~\cite[Thm. 4.3.1]{Ichiishi_1983}.
This requires many restrictive assumptions like convexity and finite-dimensionality on the decision spaces $E_i$ and the profit functions $\theta_i$.
To obtain a more general result, we take a different approach.
For that we briefly review the relevant topological notions before stating our main result.

A metrizable space $X$ is an \textit{absolute neighborhood retract}, or ANR for short, if for any embedding $X\hookrightarrow Y$ of $X$ into a metrizable topological space $Y$ as a closed subspace, the subspace $X\subseteq Y$ is a neighborhood retract of $Y$.
For example, the class of locally compact simplicial complexes belongs to the class of ANRs.
Furthermore, any Euclidean neighborhood retract (ENR) is an ANR in the class of separable metric spaces and in the class of normal spaces.
For more insight on ANRs see for example \cite{Sze-tsen1965, Borsuk1967, Masrdesic1999}.

Let $\Bbbk$ be an arbitrary field. We say that a topological space $X$ is $\Bbbk$-acyclic with respect to a specified homology theory $H_{*}$ if its reduced homology with coefficients in $\Bbbk$ vanishes, i.e., $\widetilde{H}_{*}(X;\Bbbk) \cong 0$.
For example, all classifying spaces of finite groups are $\Bbbk$-acyclic with respect to the singular homology theory and for an appropriately chosen field $\Bbbk$, e.g., the real or rational numbers.
In particular, $\widetilde{H}_*(\R\mathrm{P}^{\infty};\mathbb{F}_3)\cong 0$ while $\widetilde{H}_i(\R\mathrm{P}^{\infty};\mathbb{F}_2)\cong \mathbb{F}_2$ for all $i\geq 1$.
In our main result we assume the Vietoris homology theory introduced by Leopold Vietoris in his 1927 paper \cite{Vietoris1927}.
For a detailed development of this homology theory consult the work of Robert Reed \cite{Reed1980}.
As pointed out in \cite[Prop.\,2.2]{Tesfatsion1983}, we can consider the singular homology theory instead.

\medskip
With both the game-theoretical and topological definitions laid out, we now state our main result on the existence of Nash equilibria in abstract economies.
For the sake of brevity, we only present the main version along with its most notable corollaries here.
In this form, it extends the result on non-cooperative games of Leigh Tesfatsion from her 1975 PhD Thesis, published in~\cite{Tesfatsion1983}.
We will also discuss alternate versions of the main result, as well as many corollaries which, on the one hand, re-prove known results, and on the other hand, may be more convenient to use in practice.

We say that a set-valued map $F\colon X\longrightarrow 2^Y$ is \textit{strict} if for all $x\in X$ the value $F(x)\neq\emptyset$ is non-empty, following the conventions from Aubin and Frankowska's book~\cite{Aubin_Frankowska_2009}. The precise definitions of the remaining notions, can also be found in the same book~\cite[Chap.\,1.3,\,1.4]{Aubin_Frankowska_2009}.
\begin{theorem}\label{thm:MainResult}
	Let $\Bbbk$ be a fixed field and let $N\geq 1$ be an integer. Suppose that $\mathcal{E}=(E_i, X_i, \theta_i\colon 1\leq i \leq N)$ is an abstract economy such that
	\begin{compactenum}[\rm \quad (1)]
		\item the set valued map $X\colon E\longrightarrow 2^{E}$ is strict, compact-valued and both upper and lower semicontinuous,
		\item the function $\theta_i\colon\operatorname{Graph}(X_i)\longrightarrow\mathbb{R}$ is continuous for all $1\leq i \leq N$, and
		\item $E$ is a compact, $\Bbbk$-acyclic ANR and $\Phi(x)$ is $\Bbbk$-acyclic for every $x\in E$,
	\end{compactenum}
	where $\Phi$ denotes the global best-reply correspondence of $\mathcal{E}$.

	\noindent Then $\mathcal{E}$ has at least one Nash equilibrium.
\end{theorem}

\subsection{Examples}

In order to demonstrate that our new result is a genuine generalization, we provide constructions of abstract economies which previous existence results don't apply to due to an absence of convexity properties, but can be treated using our new result.

\medskip
We begin with a simple example which demonstrates that we are now able to treat more general profit functions with our new result.

\begin{example}
	Consider a $2$-player abstract economy in admissible form given by $E_1 = E_2 = [0, 1]^2$ and $C=D^2$, the unit disk in $\mathbb{R}^2$, which is a subset of the cube $[0, 1]^2$. Let $\theta_1(x, y) = -\min\{|x_1-y_1|, |x_2-y_2|\}$, and let $\theta_2$ be a quasi-concave continuous function on $\mathbb{R}^2\times\mathbb{R}^2$. This implies that the best-reply correspondence $\Phi_2$ is convex-valued.
	As for the best-reply correspondence $\Phi_1$, it is a star-like shape centered at $y$. So while it is not convex, it is contractible and therefore acyclic. Consequently, Theorem~\ref{thm:MainResult} is still applicable, proving that the abstract economy just described has a Nash equilibrium.
\end{example}

\begin{example}
	Consider a two-player game with strategy spaces $E_1 = E_2 = D(\xi)$, where $\xi$ is a Euclidean real vector bundle of dimension at least $2$ over the real projective space $\mathbb{R}\mathbb{P}^{2n}$ and $D(\xi)$ denotes the total space of its unit disk bundle. Consider the set-valued maps $X_1\colon E\longrightarrow E_1$ and $X_2\colon E\longrightarrow E_2$ given by
	\begin{displaymath}
		X_1((\ell_1, v_1), (\ell_2, v_2)) :=  \{ (\ell_1, w) : \delta \cdot e^{-\lVert v_2 \rVert} \leq \lVert w \rVert \leq e^{-\lVert v_2 \rVert}\},
	\end{displaymath}
	where $\delta\in(0, 1)$ is a constant. Similarly,
	\begin{displaymath}
		X_2((\ell_1, v_1), (\ell_2, v_2)) := \{ (\ell_2, w) : \varepsilon \cdot e^{-\lVert v_1 \rVert} \leq \lVert w \rVert \leq e^{-\lVert v_1 \rVert}\},
	\end{displaymath}
	where $\varepsilon\in(0, 1)$ is a constant. Consider the profit functions
	\begin{displaymath}
		\theta_1\colon \operatorname{Graph}(X_1) \longrightarrow\mathbb{R} \qquad \text{and} \qquad \theta_2\colon \operatorname{Graph}(X_2) \longrightarrow\mathbb{R}
	\end{displaymath}
	defined by
	\begin{displaymath}
		\theta_1((\ell_1, v_1), (\ell_2, v_2))  = \frac{\langle v_1, v_2 \rangle}{\lVert v_1\rVert\lVert v_2 \rVert}, \qquad
		\theta_2((\ell_1, v_1), (\ell_2, v_2))  = \frac{\langle v_1, v_2 \rangle}{\lVert v_1\rVert\lVert v_2 \rVert}.
	\end{displaymath}
	In global strategies relevant for the equilibrium, we cannot have the denominator equal $0$. So for formal completeness, in the case that one of the vectors is zero, we define the profit function to be $0$. Though this does not matter for our analysis.
	Consequently, the marginal functions evaluate to $1$, because the players are maximizing the cosine of an angle between nonzero vectors. Hence, the best-reply correspondence of the first player is:
	\begin{align*}
		\Phi_1((\ell_1, v_1), (\ell_2, v_2)) & = \Big\{(\ell, u) \in X_1((\ell_1, v_1), (\ell_2, v_2)) : \frac{\langle u, v_2 \rangle}{\lVert u\rVert\lVert v_2 \rVert} = 1 \Big\}                                                   \\
		                                     & = \Big\{(\ell_1, u) : \delta \cdot e^{-\lVert v_2 \rVert} \leq \lVert u \rVert \leq e^{-\lVert v_2 \rVert}, \frac{\langle u, v_2 \rangle}{\lVert u\rVert\lVert v_2 \rVert} = 1 \Big\} \\
		                                     & = \Big\{(\ell_1, \lambda v_2) : \frac{\delta \cdot e^{-\lVert v_2 \rVert}}{\lVert v_2\rVert} \leq \lambda \leq \frac{e^{-\lVert v_2 \rVert}}{\lVert v_2\rVert}  \Big\},
	\end{align*}
	which is just a closed interval, and therefore contractible. Analogously, the best-reply correspondence of the second player is also a closed interval.

	Since the total space of the disk bundle has contractible fibers, it is homotopy equivalent to the base space $\mathbb{R}\mathbb{P}^{2n}$, which in this case is for example $\mathbb{F}_3$-acyclic.

	Thus, the assumptions of Theorem~\ref{thm:MainResult} are satisfied, implying that the abstract economy we just constructed has a Nash equilibrium.
\end{example}

\begin{example}
	Our next example is a $2$-player abstract economy in admissible form and demonstrates an application of Theorem~\ref{thm:MainResultAdmissibleForm}. Let $n\geq 1$ be an integer and let $\xi$ be a vector bundles over $\mathbb{R}\mathbb{P}^{2n}$. Set $E_1=E_2= E(\xi)$ to be the total space of a Euclidean real vector bundle $\xi$.

	Define the set of admissible strategies to be $C:= D(\xi\oplus \xi)$, where $\xi\oplus \xi$ is Whitney sum, and $D$ stands for the total space of the disk bundle.
	More precisely, the total space of the disk bundle, as a space, is given as follows
	\begin{displaymath}
		D(\xi\oplus \xi) = \{ ((\ell, v_1), (\ell, v_2)) \in E(\xi\oplus \xi)  \subseteq E_1 \times E_2 : \lVert v_1 \rVert^2 + \lVert v_2 \rVert^2 \leq 1 \}.
	\end{displaymath}
	Intuitively speaking, we require both players to play in fibers over the same line and the sum of the norms of their choices must be at most $1$.

	The profit functions of the players are given by
	\begin{displaymath}
		\theta_1((\ell_1, v_1), (\ell_2, v_2)) := \langle v_1, v_2\rangle + \lVert v_1\rVert, \qquad \theta_2((\ell_1, v_1), (\ell_2, v_2)) := \langle v_1, v_2\rangle + \lVert v_2\rVert.
	\end{displaymath}
	With the strategy spaces, set of admissible strategies, and profit functions defined, we are ready to apply our result on abstract economies, Theorem~\ref{thm:MainResultAdmissibleForm}.

	\medskip
	First, $C$ is compact, as a fiber bundle with compact fibers over a compact base space. Next, the local triviality of the disk bundle together with Theorem~\ref{thm:LCTVSANRconditionsPalais} shows that $C$ is an ANR. Since the fibers are contractible, we know that $C = D(\mathbb{R}\mathbb{P}^{2n}) \simeq \mathbb{R}\mathbb{P}^{2n}$ is homotopy equivalent to real projective space of even dimension, which is acyclic e.g., with respect to the finite field with three elements $\mathbb{F}_3$.

	Next, the best-reply correspondence $\Phi_1$ is just single point valued.
	This follows immediately from its definition and the observation that the maximum of the scalar product, when $v_2$ is fixed, is attained whenever $v_1 =\lambda v_2$ for $\lambda > 0$, and the maximum of the norm is attained whenever $\lVert v_1 \rVert^2 + \lVert v_2 \rVert^2 = 1$.
	There is only one vector $v_1$ satisfying both of these requirements, for any fixed $v_2$. The same, symmetric argument applies for $v_2$.

	It remains only to verify that the induced admissible set-valued maps $X_i\colon E\longrightarrow 2^{E_i}$, $1\leq i \leq 2$, defined by the admissible set $C = D(\xi\oplus \xi)$ are continuous as set-valued maps.

	Upper semicontinuity is evident: It follows from the Closed Graph Lemma~\ref{lem:ClosedGraphLemma} and the fact that $C$, the graph of $X$ is compact, so in particular closed.

	Lower semicontinuity is verified by definition. For that, recall that in an abstract economy in admissible form the set-valued maps $X_i\colon E\longrightarrow 2^{E_i}$, $1\leq i \leq 2$, are given by
	\[
		X_i(x):= \{y_i\in E_i\colon (x_1,\dots,x_{i-1}, y_i, x_{i+1}, \dots, x_N)\in C\},
	\]
	where $x\in E$.
	Take $i=1$ since the case $i=2$ is analogous.

	Let $((\ell, v_1), (\ell, v_2)) \in C$ be arbitrary and let $\{(\ell^{(n)}, v_2^{(n)})\}_{n\geq 1} \to (\ell, v_2)$ be an arbitrary sequence in $E_2$ converging to $(\ell, v_2)$.
	Now, we have to produce a converging sequence $\{(\ell^{(n)}, v_1^{(n)})\}_{n\geq 1} \to (\ell, v_1)$ with the property that for all $n\geq 1$ we have $(\ell^{(n)}, v_1^{(n)})\in X_1(\ell^{(n)}, v_2^{(n)})$.

	As indicated by our notation, the requirement that $(\ell^{(n)}, v_1^{(n)})\in X_1(\ell^{(n)}, v_2^{(n)})$ forces us to consider the  sequence $\{\ell^{(n)}\}_{n\geq 1}$ as the desired sequence in $E_2$.
	For the choice of the sequence $\{v_2^{(n)}\}_{n\geq 1}$, local triviality of the bundle $\xi$ comes into play.
	By selecting a suitable trivializing neighborhood of $(\ell, v_2)$ in $E(\xi)$ we may assume that the  sequence $\{(\ell^{(n)}, v_2^{(n)})\}_{n\geq 1}$ belongs to the neighborhood $D^n\oplus D^m$ of  $(\ell, v_2)$.
	Here $D^m$ is the fiber of the disk bundle over $\ell$.
	For our sequence, select $v_1^{(n)}\in D^m$ to be the closest point to $v_1$ such that $\lVert v_1^{(n)}\rVert \leq 1-\lVert v_2^{(n)}\rVert$, to ensure that the sequence remains inside $C$. The reason for not taking just a constant sequence $v_1^{(n)}=v_1$ is due to the definition of $C$, that is, $ \lVert v_1 \rVert^2 + \lVert v_2 \rVert^2 \leq 1$.

	\medskip
	In conclusion, we've verified all the requirements of Theorem~\ref{thm:MainResult}, which shows that the abstract economy $\mathcal{E}$ defined in this example has a Nash equilibrium. Note that our strategy spaces are not even contractible, so results relying on convexity properties would not be applicable.
\end{example}

\subsection{Known results as corollaries}

Let us establish that Theorem~\ref{thm:MainResult} does, in fact, include known results as special cases. To begin, we present the famous existence result from Ichiishi~\cite[Thm.\,4.3.1]{Ichiishi_1983}, paraphrased in the notation we've introduced so far.

\begin{corollary}\label{cor:Ichiishi_1983}
	Let $N\geq 1$ be an integer and let $\mathcal{E}=(E_i,X_i,\theta_i : 1\leq i\leq N)$ be an abstract economy.
	If
	\begin{compactenum}[\quad \rm(1)]
		\item  $E_i\neq \emptyset$ is a convex and compact subset of a Euclidean space for all $1\leq i \leq N$,
		\item  $X_i\colon E\longrightarrow 2^{E_i}$ is strict, both upper and lower semicontinuous, and convex-valued for all $1\leq i \leq N$,
		\item $\theta_i\colon\operatorname{Graph}(X_i)\longrightarrow\R$ is continuous for all $1\leq i \leq N$, and
		\item the restriction function $\theta_i|_{\{x\}\times X_i(x)}\colon \{x\}\times X_i(x)\longrightarrow \R$ given by $(x,y)\longmapsto \theta_i(x,y)$ is quasi-concave for every $1\leq i \leq N$ and every $x\in E$,
	\end{compactenum}
	then the abstract economy $\mathcal{E}$ has a Nash equilibrium.
\end{corollary}

For completeness' sake let us mention that a real function $f\colon V\longrightarrow\mathbb{R}$, defined on a convex subset $V$ of a Euclidean space, is called \textit{quasi-concave}, if for all $x,y\in V$ and all $\lambda\in[0,1]$ we have that $f(\lambda x + (1-\lambda)y) \geq \min\{f(x),f(y)\}$. In particular, all concave functions are also quasi-concave.

\medskip
Many special instances of the following corollary are central to modern applications in the realm of constrained multicriteria optimization, and in particular optimization in the context of partial differential equations.
In these applications, subspaces of function spaces are considered, which is why we address them specifically.

\begin{corollary}\label{cor:FunctionSpacesCorollary}
	Let $N\geq 1$ be an integer and let $\mathcal{E}=(E_i,X_i,\theta_i : 1\leq i\leq N)$ be an abstract economy.
	If
	\begin{compactenum}[\quad \rm(1)]
		\item  $E_i\neq \emptyset$ is a metrizable compact convex subset of a locally convex topological vector space for every $1\leq i \leq N$,
		\item  $X_i\colon E\longrightarrow 2^{E_i}$ is strict, both upper and lower semicontinuous and closed-valued for every $1\leq i \leq N$,
		\item $\theta_i\colon\operatorname{Graph}(X_i)\longrightarrow\R$ is continuous for all $1\leq i \leq N$, and
		\item the restriction function $\theta_i|_{\{x\}\times X_i(x)}\colon \{x\}\times X_i(x)\longrightarrow \R$ given by $(x,y)\longmapsto \theta_i(x,y)$ is quasi-concave for every $1\leq i \leq N$ and every $x\in E$,
	\end{compactenum}
	then the abstract economy $\mathcal{E}$ has a Nash equilibrium.
\end{corollary}

A further specialization of this corollary arises when all $E_i$ are closed, bounded and convex subspaces of separable reflexive Banach spaces $V_i$.
In this case we consider $E_i$ to have the subspace topology of the weak topology on $V_i$.
Topologized in this way, $E_i$ are compact, metrizable and locally convex.
If we add the assumption that $X_i$ are convex-valued, then they also are compact-valued with respect to the weak topology. Assuming $X_i$ to be upper and lower semicontinuous with respect to the weak topology and that $\theta_i$ are continuous with respect to the weak topology, we arrive within the scope of Corollary~\ref{cor:FunctionSpacesCorollary}.

With this observation, we have recovered numerous state-of-the-art existence results, like \cite[Thm.\,2.3]{Hintermueller2013}, \cite[Thm.\,3.4]{Hintermueller2015}, \cite[Sec.\,2]{Hintermueller2022}, \cite[Thm.\,5]{Hintermueller2023}, and \cite[Thm.\,4.1]{Hintermuller2023preprint}.
Consequently, we have significantly expanded all known Nash equilibrium existence results.

The compactness of all the spaces $E_i$, and consequently of the space $E$ is a critical assumption for any result aiming to establish the existence of a Nash equilibrium. Though at first glance some results may look like they do not assume the compactness of $E$, it always comes as a consequence of the remaining assumptions.

In fact, we may construct a very simple symmetric two-player zero-sum non-cooperative game where $E$ is not compact which does not have a Nash equilibrium.
Consider the two-player non-cooperative game called the \textit{Bigger Number Game}, given by $E_1=E_2=\mathbb{N}$ with the discrete topology, and where the profit functions $\theta_i\colon E\times E_i\longrightarrow \mathbb{R}$ are defined as follows:
\begin{displaymath}
	\theta_1((x_1,x_2),y_1) =  \begin{cases}
		1,           & y_1 > x_2, \\
		\frac{1}{2}, & y_1 = x_2, \\
		0,           & y_1 < x_2,
	\end{cases} \qquad \text{and} \qquad
	\theta_2((x_1,x_2),y_2) =  -\theta_1((x_1,y_2), x_1).
\end{displaymath}
In words, both players select a natural number. The player with the larger number gets a profit of $1$, the other $0$. If the numbers are equal both players get a profit of $\tfrac{1}{2}$.
In this game, the sets $E_1,E_2=\mathbb{N}$ are not compact. Thanks to the discrete topology on $\mathbb{N}$, the profit functions are indeed continuous.
Clearly this game does not have a Nash equilibrium, since one player would always profit from changing his chosen number to be larger than the opponent's number. More precisely, if the global strategy $(x_1,x_2)$ is chosen, and if for example $x_1\leq x_2$, then player $1$ could unilaterally increase his profit by choosing any number $x_1' > x_2$.

\medskip This paper builds on the results and considerations in the Master's Thesis~\cite{Blagojevic2024}.

\medskip\noindent
\textbf{Acknowledgements.}
The first author would like to thank Dirk Werner for numerous discussions, suggestions and his continued support as well as Volker John for his support.

\medskip

\section{From a GNEP to a coincidence problem}

A key step for proving any result on the existence of a Nash equilibrium, so in particular for proving our main result too, is developing some kind of existence criterion which is more approachable than the definition of a Nash equilibrium itself. In our case, the criterion will be phrased in terms of the coincidence of two particular maps. And then, we will show how we can get an even simpler criterion for a special subclass of abstract economies.

\subsection{Coincidence criterion for abstract economies}
As preparation, we first collect some definitions we've encountered so far in one place for easier reference and precision in exposition.

\begin{definition}\label{def:AbstractEconomyAndRelatedNotions}
	An $N$-player \textit{abstract economy} $\mathcal{E}=(E_i,X_i,\theta_i\colon 1\leq i\leq N)$ is given as a collection of the following data:

	\begin{compactitem}[\quad --]
		\item $(E_1,\dots,E_N)$ is a family of non-empty metric spaces, with $E:=E_1\times \cdots\times E_N$,
		\item $(X_1,\dots,X_N)$ is a family of set-valued maps where $X_i\colon E\longrightarrow 2^{E_i}$. Additionally, we denote by $X\colon E\longrightarrow 2^E$ the product set-valued map $X(x):= X_1(x)\times\cdots\times X_N(x)$,
		\item $(\theta_1,\dots,\theta_N)$ is a family of continuous functions $\theta_i\colon\operatorname{Graph}(X_i)\longrightarrow\R$, for $1\leq i\leq N$.
	\end{compactitem}

	\medskip\noindent
	A global decision $x \in E$ is called \textit{social equilibrium} or \textit{Nash equilibrium} if
	\begin{compactitem}[\quad --]
		\item for all $1\leq i \leq N$ we have $x_i\in X_i(x)$, and
		\item for all $1\leq i \leq N$ and all $x_i'\in X_i(x)$ it holds that
		\begin{displaymath}
			\theta_i(x,x_i) \geq \theta_i(x,x_i').
		\end{displaymath}
	\end{compactitem}

	\medskip\noindent
	The \textit{best-reply correspondence of player} $i$ is the set-valued map $\Phi_i\colon E\longrightarrow 2^{E_i}$, given by:
	\begin{alignat*}{3}
		\varphi_i\colon & E\longrightarrow\mathbb{R}\cup\{\infty\}, & \qquad & x \longmapsto  \sup\{\theta_i(x, y_i) \colon y_i\in X_i(x)\},            \\
		\Phi_i\colon    & E\longrightarrow 2^{E_i},                 & \qquad & x \longmapsto    \{y_i\in X_i(x)\colon \theta_i(x,y_i) = \varphi_i(x)\}.
	\end{alignat*}
	We denote by $\Phi(x):=\Phi_1(x)\times\cdots\times\Phi_N(x)\colon E\longrightarrow 2^{E}$ the so-called \textit{global best-reply correspondence}. Recall that the functions $\varphi_i$ are often called the \textit{marginal functions} of $\theta_i$.
\end{definition}

We have already noted that the Nash equilibrium condition on a global strategy $x\in E$ is equivalent to the requirement that $x\in\Phi(x)$, which suggests an application of fixed-point theorems for establishing the existence of a Nash equilibrium.

Instead, let us briefly investigate the relations between the defined objects.
We have the graph $\operatorname{Graph}(\Phi)\subseteq E\times E$ and the (restrictions of the) canonical projections $\pi_1,\pi_2\colon\operatorname{Graph}(\Phi)\longrightarrow E$.
Assume for a moment that $\pi_1$ and $\pi_2$ have a coincidence, that means there exists a point $(x,y)\in\operatorname{Graph}(\Phi)$ such that $\pi_1(x,y) = \pi_2(x,y)$.
In other words, $x=y$, which yields a point $(x,x)\in\operatorname{Graph}(\Phi)$, or equivalently the fixed point $x\in \Phi(x)$.
This, in turn, is equivalent to $x$ being a Nash equilibrium of $\mathcal{E}$.
In conclusion, we've just proven the following criterion for the existence of a Nash equilibrium in an abstract economy.

\begin{lemma}[Coincidence criterion]\label{lem:CoincidenceCriterion1}
	Let $N\geq 1$ be an integer, $\mathcal{E}=(E_i,X_i,\theta_i : 1\leq i\leq N)$ an abstract economy and let $\Phi$ be the global best-reply correspondence of $\mathcal{E}$. Denote by
	\begin{displaymath}
		\pi_1 \colon \operatorname{Graph}(\Phi)\longrightarrow E \quad \text{and}\quad \pi_2 \colon \operatorname{Graph}(\Phi)\longrightarrow E
	\end{displaymath}
	the restrictions of the canonical projections $E\times E\longrightarrow E$. Then $\mathcal{E}$ has a Nash equilibrium if and only if $\pi_1$ and $\pi_2$ have a coincidence on $\operatorname{Graph}(\Phi)$.
\end{lemma}

\subsection{Nikaido--Isoda Function}\label{subsec:NikaidoIsodaFunction}

The global best-reply correspondence $\Phi\colon E\longrightarrow 2^{E}$ is a key object for describing the nature of an abstract economy, as it encapsulates which decisions perfectly rational players would make as reactions to any decisions by their opponents.
Unfortunately, it is quite cumbersome to explicitly compute, especially since it requires knowing the $N$ best-reply correspondences of the individual players, which require understanding all $N$ profit functions.

One very helpful idea is given by Hukukane Nikaido and Kazuo Isoda in their paper ~\cite{NikaidoIsoda1955} from 1955.
Considering the timeline of when game theory was developed, their paper is one of the earliest and contains many influential observations related to non-cooperative games.
The authors associate to a non-cooperative game a particular real-valued function, which we now call Nikaido--Isoda Function, serving as a unification of the $N$ different profit functions.
It turns out that the construction is actually transferrable to abstract economies as well.

A tiny concession has to be made to make room for the Nikaido--Isoda Function.
In the Definition~\ref{def:AbstractEconomyAndRelatedNotions} we require that $\theta_i\colon\operatorname{Graph}(X_i)\longrightarrow\mathbb{R}$ are defined on the graph of the corresponding set-valued map $X_i$.
Now it will be necessary to assume that for all $1\leq i\leq N$ the profit function $\theta_i$ admits a continuous extension to $E\times E_i \supseteq \operatorname{Graph}(X_i)$.
In the vast majority of cases, a straightforward application of the Tietze--Urysohn Extension Theorem~\cite[Thm.\,35.1]{Munkres2000} immediately yields the sought extension.
With this, we define the Nikaido--Isoda Function as follows.

\begin{definition}\label{def:NikaidoIsodaFunction}
	Let $\mathcal{E}=(E_i,X_i,\theta_i\colon 1\leq i\leq N)$ be an abstract economy with continuous profit functions $\theta_i\colon E\times E_i\longrightarrow\mathbb{R}$. The (continuous) function $\Psi\colon E\times E\longrightarrow\mathbb{R}$ given by
	\begin{displaymath}
		\Psi(x,y)=\sum_{i=1}^{N}\big(\theta_i(x,y_i)-\theta_i(x,x_i)\big)
	\end{displaymath}
	is called the \textit{Nikaido--Isoda Function} of $\mathcal{E}$.
\end{definition}

We apply the marginal function construction to the Nikaido--Isoda Function to obtain a marginal function $V$ and the related correspondence $R$:

\begin{alignat*}{3}
	V\colon & E\longrightarrow\mathbb{R}\cup\{\infty\}, & \qquad & x \longmapsto  \sup\{\Psi(x, y) \colon y\in X(x)\},    \\
	R\colon & E\longrightarrow 2^{E},                   & \qquad & x \longmapsto    \{y\in X(x)\colon \Psi(x,y) = V(x)\}.
\end{alignat*}

It may not be apparent from the deliberations above, but in fact, we have just once again obtained the global best-reply correspondence.
That means the set-valued maps $\Phi$ and $R$ coincide.
As far as the author can tell, this was not pointed out in earlier sources.

\begin{proposition}\label{prop:NI_brings_nothing_new}
	Let $\mathcal{E}= (E_i,X_i,\theta_i : 1\leq i\leq N)$ be an abstract economy with $\theta_i\colon E\times E_i\longrightarrow\mathbb{R}$.
	Using the notation above, and denoting the global best-reply correspondence of $\mathcal{E}$ with $\Phi$ we have that $\Phi(x) = R(x)$ for every fixed point $x\in X(x)$.
\end{proposition}
\begin{proof}
	Instead of proving $\Phi(x)=R(x)$ we prove the equivalent statement that $X(x)\setminus\Phi(x)=X(x)\setminus R(x)$.

	\medskip\noindent
	First, assume that $y\in X(x)\setminus R(x)$.
	Consequently, by definition of $R(x)$ and $V(x)$, there exists $y'\in X(x)$ such that $\Psi(x,y') > \Psi(x,y)$, that is
	\[
		\sum_{i=1}^{N}\bigl(\theta_i(x,y'_i)-\theta_i(x,x_i)\bigr) > \sum_{i=1}^{N}\bigl(\theta_i(x,y_i)-\theta_i(x,x_i)\bigr) \quad \Longrightarrow \quad
		\sum_{i=1}^{N}\theta_i(x,y'_i) > \sum_{i=1}^{N}\theta_i(x,y_i).
	\]
	Hence, there exists at least one $1\leq j\leq N$ such that $\theta_j(x,y'_j) > \theta_j(x,y_j)$ and so $y_j\not\in \Phi_j(x)$ implying that  $y\in X(x)\setminus\Phi(x)$.

	\medskip
	Now assume that $y \in X(x)\setminus\Phi(x)$, or in other words there is an index $1\leq j\leq N$ such that $y_j\not\in \Phi_j(x)$.
	Consequently, there is $y_j'\in X_j(x)$ such that $\theta_j(x,y_j') > \theta_j(x,y_j)$ and so
	\begin{multline*}
		\Psi(x,y) = \Psi((x_1,\dots,x_j,\dots,x_N),(y_1,\dots,y_j,\dots,y_N)) = \sum_{i=1}^{N}\bigl(\theta_i(x,y_i)-\theta_i(x,x_i)\bigr) \\ =
		\sum_{i=1}^{j-1}\bigl(\theta_i(x,y_i)-\theta_i(x,x_i)\bigr) + (\theta_j(x,y_j)-\theta_j(x,x_j)) + \sum_{i=j+1}^{N}\bigl(\theta_i(x,y_i)-\theta_i(x,x_i)\bigr) \\ < \sum_{i=1}^{j-1}\bigl(\theta_i(x,y_i)-\theta_i(x,x_i)\bigr) + (\theta_j(x,y_j')-\theta_j(x,x_j)) + \sum_{i=j+1}^{N}\bigl(\theta_i(x,y_i)-\theta_i(x,x_i)\bigr)   \\ =  \Psi((x_1,\dots,x_j,\dots,x_N),(y_1,\dots,y_j',\dots,y_N)).
	\end{multline*}
	Therefore, $V(x)\neq\Psi(x,y)$ and so $x\in X(x)\setminus R(x)$.
\end{proof}

Combining this result with the fact that in an abstract economy, Nash equilibria are fixed points of the global best-reply correspondence, we arrive at the following corollary.

\begin{corollary}\label{cor:FixedPointNIBestReplyCorrespondence}
	Let $\mathcal{E}=(E_i,X_i,\theta_i\colon 1\leq i\leq N)$ be an abstract economy with $\theta_i\colon E\times E_i\longrightarrow\mathbb{R}$. Then $x\in X(x)$ is a Nash equilibrium of $\mathcal{E}$ if and only if $x\in R(x)$.
\end{corollary}

Notice that by definition of the Nikaido--Isoda Function, we have that
\begin{displaymath}
	V(x) \overset{\text{def}}{\geq}\Psi(x,x) = \sum_{i=1}^{N}\left(\theta_i(x,x_i) - \theta_i(x,x_i)\right) = 0.
\end{displaymath}

Consequently, $V(x)\geq 0$, and if we have that $x\in R(x)$, meaning $\Psi(x,x) = V(x)$, we must have $V(x) = 0$. So for any Nash equilibrium $x$ we have that $V(x) = 0$. Is the converse true? Mostly yes, however, a mild additional assumption is necessary. Recall that for all $1\leq i \leq N$ we allow $X_i\colon E\longrightarrow 2^{E_i}$ to depend on the $i$-th factor of $x\in E$. Intuitively, this is an unnatural assumption, and, unfortunately, we have to give it up in order to gain the converse direction.

\begin{proposition}\label{prop:NikaidoIsodaMaximizationFindsGNEP}
	Let $\mathcal{E}=(E_i,X_i,\theta_i\colon 1\leq i\leq N)$ be an abstract economy with $\theta_i\colon E\times E_i\longrightarrow\mathbb{R}$ and retain the notation used so far. Assume that for all $1\leq i\leq N$, for any $x\in E$ and any $y_i\in E_i$ we have that
	\begin{displaymath}
		X_i(x_1,\dots,x_i,\dots,x_N) = X_i(x_1,\dots,y_i,\dots,x_N).
	\end{displaymath}
	In words: Assume that $X_i(x)$ for $x=(x_1,\dots,x_N)$ does not depend on $x_i$.
	Then for all $x\in X(x)$:
	\begin{displaymath}
		V(x) = 0 \quad \Longleftrightarrow \quad \text{$x$ is a Nash equilibrium of $\mathcal{E}$}.
	\end{displaymath}
\end{proposition}
\begin{proof}
	It remains to show that $V(x)=0$ implies that $x$ is a Nash equilibrium of $\mathcal{E}$. Select any player $1\leq i \leq N$ and any $y_i \in X_i(x)$. Denote $x^* = (x_1,\dots,x_{i-1},y_i,x_{i+1},\dots,x_N)$. By our assumption that $X_i(x)$ does not depend on $x_i$, we conclude for $x^*$ that also $x^*\in X(x)$. With this we may apply the definition of $V(x)$ to obtain that
	\begin{multline*}
		0\geq \Psi(x,x^*)=\\
		\sum_{j=1}^{i-1}\big(\theta_j(x,x_j) - \theta_j(x,x_j)\big)
		+\big(\theta_i(x,y_i) - \theta_i(x,x_i)\big)+\sum_{j=i+1}^{N}\big(\theta_j(x,x_j) - \theta_j(x,x_j)\big) =\\ \theta_i(x,y_i) - \theta_i(x,x_i).
	\end{multline*}
	Thus, $\theta_i(x,y_i)\leq \theta_i(x,x_i)$. Since $y_i\in X_i(x)$ and $1\leq i \leq N$ was arbitrary, we have verified that $x$ is a Nash equilibrium of $\mathcal{E}$
\end{proof}

This result is more useful if our interest is to numerically approximate $\Psi$ and consequently find the minima $V$. We know that the minima of $V$ are at $0$, and that these minima are exactly the Nash equilibria. In particular examples we may apply numerical tools to approximate these minima which are simultaneously zeros and compute the equilibria.

\medskip
For us there is yet another beneficial use of the Nikaido--Isoda function, which arises when working with a particular specialization of an abstract economy. We cover this specialization in the next section.

\subsection{Coincidence Criterion for Abstract Economies in Admissible Form}

The condition for $x$ to be a Nash equilibrium of an abstract economy $\mathcal{E}=(E_i,X_i,\theta_i\colon 1\leq i\leq N)$ requires that $x$ must in particular fulfill $x\in X(x)$.
So the set $C:= \{x\in E\colon x\in X(x)\}$ of fixed points of $X$ is an essential object for the analysis of an abstract economy.

Recall that the set-valued maps $X_1,\dots,X_N$ model the fact that there may be a restriction on the global decision chosen by the players.
Thereby we take into account the possibility that the restrictions may depend on the choices of the players.
This offers an incredibly broad framework for modelling many real processes as abstract economies, but, on the other hand, the full generality of this setup is somewhat uncommon.
More frequently, the restrictions on the global decision are given, for example, via physical constraints, which are fixed with respect to the choices of the players. This is made more precise by the following specialization of an abstract economy.

\begin{definition}\label{def:AbstractEconomyAdmissibleForm}
	An $N$-player abstract economy $\mathcal{E}=(E_i,X_i,\theta_i\colon 1\leq i\leq N)$ is said to be in \textit{admissible form} if there exists $C\subseteq E$ such that
	\begin{displaymath}
		X_i(x) = \{y_i\in E_i\colon (x_1,\dots,x_{i-1}, y_i, x_{i+1}, \dots, x_N)\in C\}.
	\end{displaymath}
	We call $C$ the set of \textit{admissible global strategies}.
\end{definition}
In the literature, abstract economies in admissible form are often referred to as \textit{jointly constrained} generalized Nash equilibrium problems, see for example~\cite{Stengl2021}.

Note that in this setup, our notation is chosen to be consistent, because it is still the case that $C = \{x\in E\colon x\in X(x)\}$ is the fixed-point set of $X$. However, there may be many different families of set-valued maps $X_i\colon E\longrightarrow 2^{E_i}$ which yield the same fixed-point set $\{x\in E\colon x\in X(x)\}$. In the case of an abstract economy in admissible form, the family of set-valued maps $X_1,\dots,X_N$ is specified explicitly. It is worth noting that the definition of $X_i(x)$ in the case of an abstract economy is such that $X_i(x)$ does not depend on $x_i$, so we may apply Proposition~\ref{prop:NikaidoIsodaMaximizationFindsGNEP} to abstract economies in admissible form whenever necessary.

In this situation, we develop a condition for the existence of a Nash equilibrium which is even more convenient to use. To understand why it is useful, notice that in Section~\ref{subsec:NikaidoIsodaFunction} we used the marginal function construction with the Nikaido--Isoda Function to arrive at the definition $V(x) := \sup\{\Psi(x,y)\colon y\in X(x)\}$. Thus, $V(x)$ is the supremum of a function $\Psi(x,\cdot)$, which depends on $x$, over a set $X(x)$, which also depends on $x$. The following lemma holds for abstract economies in admissible form, removing the dependency on $x$ from the set over which the supremum has to be taken, albeit at a price.

\begin{lemma}\label{lem:ConditionAbstractEconomyAdmissibleForm}
	Let $\mathcal{E}=(E_i,X_i,\theta_i\colon 1\leq i\leq N)$ be an abstract economy in admissible form given by the set $C\subseteq E$. Define
	\begin{alignat*}{3}
		\widetilde{V}\colon & C\longrightarrow\mathbb{R}\cup\{\infty\}, & \qquad & x \longmapsto  \sup\{\Psi(x, y) \colon y\in C\},                \\
		\widetilde{R}\colon & C\longrightarrow 2^{C},                   & \qquad & x \longmapsto    \{y\in C\colon \Psi(x,y) = \widetilde{V}(x)\}.
	\end{alignat*}
	Then for all $x\in C$ we have that:
	\begin{compactenum}[\rm \quad (1)]
		\item $\widetilde{V}(x) = 0$ implies that $x$ is a Nash equilibrium of $\mathcal{E}$
		\item $x \in \widetilde{R}(x)$ implies that $x$ is a Nash equilibrium of $\mathcal{E}$
	\end{compactenum}
\end{lemma}

\noindent
In general the converse directions of the implications in these statements do not hold.

\begin{proof}
	For the first claim, Proposition~\ref{prop:NikaidoIsodaMaximizationFindsGNEP} says that it is sufficient to prove
	\begin{displaymath}
		\widetilde{V}(x)=0 \ \Longrightarrow \ V(x)=0.
	\end{displaymath}
	For that assume $y=(y_1,\dots,y_N)\in X(x)$ is arbitrary or equivalently, that $y_i\in X_i(x)$ for all $1\leq i \leq N$. By definition of an abstract economy in admissible form, we now know that for all players $1\leq i\leq N$
	\begin{displaymath}
		y^*:=(x_1,\dots,x_{i-1},y_i,x_{i+1},\dots,x_N)\in C.
	\end{displaymath}
	Using the definition of $\tilde{V}(x) = \sup\{\Psi(x, y) \colon y\in C\}$ and the assumption $\tilde{V}(x) = 0$, we obtain that $\Psi(x,y^*) \leq 0$, which, after unraveling the definition of $\Psi$, implies $\theta_i(x,y_i)-\theta_i(x,x_i) \leq 0$.
	Since $i$ was arbitrary, we can sum over all $1\leq i \leq N$ to obtain
	\begin{displaymath}
		\Psi(x,y)=\sum_{i=1}^{N}\big(\theta_i(x,y_i)-\theta_i(x,x_i)\big)\leq 0.
	\end{displaymath}
	Then, because $y\in X(x)$ was also arbitrary, we may conclude that $V(x)\leq 0$. But we already know that $V(x)\geq 0$ is always non-negative, so ultimately we obtain $V(x)=0$, implying that $x$ is a Nash equilibrium of $\mathcal{E}$.

	The second claim follows immediately because $x\in\tilde{R}(x)$ by definition implies that $0=\Psi(x,x)=\tilde{V}(x)$, and we can apply the first part.
\end{proof}

We may now repeat the same reasoning as for our previous coincidence criterion of Lemma~\ref{lem:CoincidenceCriterion1} to arrive at the following lemma for the existence of a Nash equilibrium in an abstract economy in admissible form.

\begin{lemma}[Coincidence criterion for admissible form]\label{lem:CoincidenceCriterionAdmissibleForm}
	Let $N\geq 1$ be an integer, $\mathcal{E}=(E_i,X_i,\theta_i : 1\leq i\leq N)$ an abstract economy in admissible form with $C\subseteq E$ being the set of admissible global strategies. Retaining the notation from Lemma~\ref{lem:ConditionAbstractEconomyAdmissibleForm}, set
	\begin{displaymath}
		\pi_1 \colon \operatorname{Graph}(\widetilde{R})\longrightarrow C \quad \text{and}\quad \pi_2 \colon \operatorname{Graph}(\widetilde{R})\longrightarrow C,
	\end{displaymath}
	to be the restrictions of the canonical projections $C\times C\longrightarrow C$. If $\pi_1$ and $\pi_2$ have a coincidence on $\operatorname{Graph}(\tilde{R})$, then $\mathcal{E}$ has a Nash equilibrium.
\end{lemma}

\section{Proof of the main result}

\subsection{Prerequisites}

We begin with some results which will be necessary for the proof of our main Nash equilibrium existence result, Theorem~\ref{thm:MainResult}. First is a famous theorem from set-valued analysis due to Claude Berge~\cite[Ch.VI.3,\,Thm.\,1]{Berge1997} from 1959, which gives us some valuable properties of the functions which arise from the marginal function construction we've been using so far. We state it in a way similar to Ichiishi~\cite[Thm.\,2.3.1]{Ichiishi_1983}.

\begin{theorem}[Berge's Maximum Theorem]\label{thm:MaximumTheorem}
	Let $X$ and $Y$ be metric spaces, the function  $f\colon X\times Y\longrightarrow \mathbb{R}$ continuous, and let $F\colon X\longrightarrow 2^Y$ be a set-valued map that is strict, compact-valued, and both lower and upper semicontinuous.
	Then
	\begin{compactenum}[\rm \quad (1)]
		\item the marginal function $\varphi\colon X\longrightarrow \mathbb{R}$ defined by $\varphi(x):=\max\{f(x,y)\colon y\in F(x)\}$ is continuous, and
		\item the marginal set-valued map $\Phi\colon X\longrightarrow 2^Y$ given by $\Phi(x):= \{y\in F(x)\colon f(x,y)=\varphi(x)\}$ is upper semicontinuous and strict.
	\end{compactenum}
\end{theorem}

\medskip
The second result we use is due to Samuel Eilenberg and Deane Montgomery \cite[Thm.\,3]{EilenbergDeane1946} in 1946.
In order to state it we need to introduce some related topological notions.

Let $X$ and $Y$ be {\em compact} metric spaces, $\Bbbk$ a field, and let $f\colon X\longrightarrow Y$ be a continuous function.
Fix a homology theory $H_*$ with coefficients in the field $\Bbbk$.
Then, the function $f$ has the \textit{Vietoris property}, with respect to the homology theory $H_*$, if for every $y\in Y$ the pre-image $f^{-1}(\{y\})$ is $\Bbbk$-acyclic, that is $H_0(f^{-1}(\{y\}))\cong\Bbbk$ and $H_i(f^{-1}(\{y\}))\cong 0$ for $i\geq 1$.
Notice that if $f$ has the Vietoris property then $f$ has to be surjective, as the empty set is not acyclic. As the next theorem shows, the Vietoris property, along with some other mild assumptions, is sufficient to guarantee the existence of a coincidence of two functions.

\begin{theorem}\label{thm:EilenbergCoincidenceCondition}
	Let $\Bbbk$ be a fixed field, $Y$ a compact, $\Bbbk$-acyclic {\rm ANR}, $X$ a compact metric space, and let $f\colon X\longrightarrow Y$ and $g\colon X\longrightarrow Y$ be continuous functions.
	Let $H_*$ be the singular homology theory with coefficients in the field $\Bbbk$.
	If (at least) one of the functions $f$ or $g$ has the Vietoris property with respect to the homology theory $H_*$, then $f$ and $g$ have a coincidence.
\end{theorem}

\medskip
In the original work of Eilenberg and Montgomery the homology theory considered was the Vietoris homology theory \cite{Vietoris1927}.
On the other hand, as demonstrated in  \cite[Prop.\,2.2]{Tesfatsion1983}, we can use the singular homology theory in the place of Vietoris homology theory when coefficients are assumed to be in a field.

The result of Eilenberg and Montgomery was generalized by Edward Begle \cite[Thm.\,2]{Begle1950},  Michael Powers \cite[Thm.\,6.3]{Powers1970}, \cite[Thm.\,5.6]{Powers1972}, Lech G\'{o}rniewicz  and Andrzej Granas \cite{GorniewiczGranas}, and Granas  and James Dugundji \cite[Cor.\,19.7.5]{GranasDugundji2003}.
These results can be used to extend Theorem \ref{thm:MainResult} even further.

\medskip
Finally, we need another lemma from set-valued analysis which relates the closedness of a graph to the properties of the set-valued map. In the literature, the following lemma is often called \textit{Closed Graph Lemma}. A proof can be found in most set-valued analysis textbooks, for example in the classical book of Jean-Pierre Aubin and H\'el\`ene Frankowska~\cite[Prop. 1.4.8]{Aubin_Frankowska_2009}.

\begin{lemma}\label{lem:ClosedGraphLemma}
	Let $F\colon X\longrightarrow 2^Y$ be a set-valued map.
	If $F$ is closed-valued, upper semicontinuous and has a closed domain in $X$, then the graph $\operatorname{Graph}(F)$ is closed in $X\times Y$.
	Additionally, if $Y$ is compact, then this statement is an equivalence.
\end{lemma}

Now we are ready to give the proof of the main result.

\subsection{Proof of Theorem~\ref{thm:MainResult}}\label{subsec:ProofOfMainResult}

We may write $X_i = p_i\circ X$, where $p_i\colon E\longrightarrow 2^{E_i}$ given by $x\longmapsto\{x_i\} $ is induced by the (continuous) canonical projection $(x_1,\dots,x_n)\longmapsto x_i$ on the $i$-th factor. Since $p_i$ is given as a singleton-valued set-valued map induced by a continuous map, it itself is both upper and lower semicontinuous. With this, we can next use the fact that the composition of upper or lower semicontinuous set-valued maps is still upper or lower semicontinuous respectively --- for proofs see e.g.~\cite[Def. 1.3.2 ff.]{Aubin_Frankowska_2009}. Consequently, for all $1\leq i \leq N$ the set-valued map $X_i$ is both upper and lower semicontinuous. This observation also immediately implies that $X_1,\dots,X_n$ are also compact-valued and strict.

From $E$ being compact, we conclude that for all $1\leq i \leq N$ the spaces $E_i$ are also compact, as continuous images of a compact set. With the upper semicontinuity of $X_i$, the conditions for the Closed Graph Lemma~\ref{lem:ClosedGraphLemma} are fulfilled, so we may conclude that $\operatorname{Graph}(X_i)$ is closed in $E\times E_i$. This allows us to apply the Tietze--Urysohn extension theorem~\cite[Thm.\,35.1]{Munkres2000} to obtain continuous extensions $\theta_i'\colon E\times E_i\longrightarrow\mathbb{R}$.

Now, with $X_i$ both upper and lower semicontinuous, $\theta_i'$ continuous on $E\times E_i$, we may apply Berge's Maximum Theorem~\ref{thm:MaximumTheorem} to the metric spaces $E$ and $E_i$ to conclude that $\Phi_i\colon E\longrightarrow 2^{E_i}$ is upper semicontinuous for all $1\leq i \leq N$. Note that it does not matter whether we're considering $\theta_i$ or the extensions $\theta_i'$, since they agree on all points which we are maximizing over.

The Closed Graph Lemma~\ref{lem:ClosedGraphLemma} now yields that $\operatorname{Graph}(\Phi)$ is closed, so also compact, as a subset of a compact space $E\times E$.

\medskip According to our first coincidence criterion~\ref{lem:CoincidenceCriterion1}, it suffices to show that the restrictions of the canonical projections
\begin{displaymath}
	\pi_1 \colon \operatorname{Graph}(\Phi)\longrightarrow E \quad \text{and}\quad \pi_2 \colon \operatorname{Graph}(\Phi)\longrightarrow E,
\end{displaymath}
have a coincidence. We wish to apply Theorem~\ref{thm:EilenbergCoincidenceCondition}, so we verify its conditions. By assumption, $E$ is a $\Bbbk$-acyclic ANR. We also have that for all $x\in E$ the set $\Phi(x)$ is $\Bbbk$-acyclic. By definition of $\pi_1$, we have that $\Phi(x)=\pi_1^{-1}(\{x\})$, proving that $\pi_1$ has the Vietoris property. Finally, we've seen that $\operatorname{Graph}(\Phi)$ is compact. All this verifies the last condition of Theorem~\ref{thm:EilenbergCoincidenceCondition}, meaning we may conclude that $\pi_1$ and $\pi_2$ have a coincidence, and, as explained, we may also conclude that $\mathcal{E}$ has a Nash equilibrium.

\qed

\subsection{Proof of Corollary~\ref{cor:FunctionSpacesCorollary}}
The following Lemma from convex analysis relating quasi-concave maps and properties of their maximum sets will be helpful. Though the proof is simple and well-known, we provide it here for completeness' sake.

\begin{lemma}\label{lem:QuasiConcaveMaximumSet}
	Let $K\subseteq X$ be a compact subset of a normed space, and let $f\colon K\longrightarrow\R$ be continuous and quasi-convex. Then the set
	\begin{displaymath}
		M_f := \operatorname{argmax}(f) := \big\{x\in K\ : f(x)=\max\{f(x') : x'\in K\} \big\}
	\end{displaymath}
	is convex.
\end{lemma}
\begin{proof}
	Choose any $y_1, y_2\in M_f$ and any $0\leq t\leq 1$.
	Set $y_t=(1-t)y_1+ty_2$.
	The convexity of $K$ implies that $y_t\in K$.
	On the other hand, the quasi-concavity of $f$ yields $f(y_t)\geq \min\{f(y_1),f(y_2)\}$.
	Since $M_f$ is defined as the set of maxima of $f$ and $y_1,y_2\in M_f$, we get that necessarily $y_t\in  M_f$.
	Hence, $M_f$ is convex.
\end{proof}

For our application of this theorem the conclusion that $M_f$ is convex is unnecessarily strong. As one may have suspected, the decisive consequence of $M_f$ being convex is actually that $M_f$ is $\Bbbk$-acyclic with respect to any field $\Bbbk$.

We also need a theorem from homotopy theory, which provides a criterion for a metrizable space to be an ANR. We state the theorem here, a proof can be found in Richard S. Palais' paper~\cite[Thm.\,5]{Palai1965} from 1965.
\begin{theorem}\label{thm:LCTVSANRconditionsPalais}
	A metrizable space is an ANR if each point has a neighborhood homeomorphic to a convex set in a locally convex topological vector space.
\end{theorem}

With the lemma and the theorem above, we are equipped to verify all the assumptions of Theorem~\ref{thm:MainResult}.

\begin{proof}
	By assumption $E_i\neq \emptyset$ are metrizable compact convex subsets of a locally convex topological vector space, so the same must hold for their product $E$. But now $E$ trivially satisfies the condition of Theorem~\ref{thm:LCTVSANRconditionsPalais}, yielding that $E$ is an ANR. As a convex set, it is contractible, so also $\Bbbk$-acyclic, for any field $\Bbbk$.

	Next, using the notation from Lemma~\ref{lem:QuasiConcaveMaximumSet}, notice that by definition of the best-reply correspondence $\Phi_i(x)$ we have that $\Phi_i(x) = M_{\theta_i(x,\cdot)}$. By assumption, for every $x\in E$ we have that $\theta_i(x, \cdot)$ is quasi-concave, implying by Lemma~\ref{lem:QuasiConcaveMaximumSet} that $\Phi_i(x)$ is convex, from which we conclude that $\Phi(x)$ is also convex, so in particular $\Bbbk$-acyclic.

	With the assumptions from the statement of the corollary and the remaining assumptions verified just now, we have completed the verification of all assumptions of Theorem~\ref{thm:MainResult}, and we may conclude that $\mathcal{E}$ has a Nash equilibrium.
\end{proof}

\subsection{Proof of Corollary~\ref{cor:Ichiishi_1983}}

Once again we verify that the conditions of our main result Theorem~\ref{thm:MainResult} are fulfilled, which is even simpler in this case.
Just like in the proof of Corollary~\ref{cor:FunctionSpacesCorollary} we again use Lemma~\ref{lem:QuasiConcaveMaximumSet} to conclude that $\Phi$ is $\Bbbk$-acyclic-valued. Finite dimensional vector spaces are in particular locally convex spaces, so the same Theorem~\ref{thm:LCTVSANRconditionsPalais} also yields that $E$ is an ANR --- though this can also be seen by simpler means. Finally, since $E$ is convex, it is also $\Bbbk$-acyclic.

The rest of the assumptions of Theorem~\ref{thm:MainResult} are provided by the assumptions of the corollary, so we may again conclude that $\mathcal{E}$ has a Nash equilibrium.
\qed

\subsection{Variant for Abstract Economies in Admissible Form}
Just like in Section~\ref{subsec:ProofOfMainResult} we may exploit our coincidence criterion Lemma~\ref{lem:CoincidenceCriterionAdmissibleForm} for abstract economies in admissible form to get a result similar to Theorem~\ref{thm:MainResult}.
\begin{theorem}\label{thm:MainResultAdmissibleForm}
	Let $\Bbbk$ be a fixed field and let $N\geq 1$ be an integer.
	Suppose that $\mathcal{E}=(E_i,X_i,\theta_i : 1\leq i\leq N)$ is an abstract economy in admissible form given by the set $C\subseteq E$ such that
	\begin{compactenum}[\rm \quad (1)]
		\item $C$ is a compact, $\Bbbk$-acyclic {\rm ANR}, and
		\item the correspondence $\widetilde{R}\colon C\longrightarrow 2^C$ is strict and $\Bbbk$-acyclic-valued.
	\end{compactenum}
	Then the abstract economy $\mathcal{E}$ has a Nash equilibrium.
\end{theorem}
\begin{proof}
	The only difference to the proof of Theorem~\ref{thm:MainResult}, is how we argue that $\widetilde{R}$ is upper semicontinuous. For this, define $Z\colon C\longrightarrow 2^C$ via $Z(x)= C$ for all $x\in C$. In other words, $Z$ is the constant set-valued map taking the value $C$. Obviously $Z$ is both upper and lower semicontinuous. Since $C$ is compact, we may now apply Berge's Maximum Theorem~\ref{thm:MaximumTheorem} to $Z$, the Nikaido--Isoda function $\Psi$ and $C\times C$. Thus, we obtain that $\widetilde{R}$ is upper semicontinuous. Consequently, the Closed Graph Lemma~\ref{lem:ClosedGraphLemma} yields that $\operatorname{Graph}(\widetilde{R})$ is closed in $C\times C$, therefore compact as a subset of a compact space $C\times C$.

	Now, we may comfortably apply the coincidence criterion for abstract economies in admissible form, Lemma~\ref{lem:CoincidenceCriterionAdmissibleForm}, which requires us to show that
	\begin{displaymath}
		\pi_1 \colon \operatorname{Graph}(\widetilde{R})\longrightarrow C \quad \text{and}\quad \pi_2 \colon \operatorname{Graph}(\widetilde{R})\longrightarrow C,
	\end{displaymath}
	have a coincidence. We know that $C$ is a $\Bbbk$-acyclic ANR, and we know that $\widetilde{R}$ is $\Bbbk$-acyclic-valued, which is equivalent to $\pi_1$ having the Vietoris property. Finally, an application of Theorem~\ref{thm:EilenbergCoincidenceCondition} completes the proof.
\end{proof}

Many abstract economies considered in applications are in admissible form, for example PDE-constrained abstract economies, which appear in the following publications:~\cite{Hintermueller2015}, ~\cite{SurowiecStaudiglKouri2023}, and ~\cite{Hintermueller2022}.

\end{document}